\documentclass[12pt]{amsart}

\usepackage{verbatim}
\usepackage{xcolor}
\usepackage{tikz}
\usepackage{tikz-cd}
\usepackage{arydshln}
\usepackage{caption}
\usepackage{subcaption}
\usepackage{graphicx}
\usepackage{ulem}
\tikzset{
	symbol/.style={
		draw=none,
		every to/.append style={
			edge node={node [sloped, allow upside down, auto=false]{$#1$}}}
	}
}

\usepackage{amsmath}
\usepackage{bbm}
\usepackage{hyperref}
\usepackage{times,amsfonts,amsmath,amstext,amsbsy,amssymb,
amsopn,amsthm,upref,eucal}
\usepackage[T1]{fontenc}

\newtheorem{theorem}{Theorem}[section]

\newtheorem{lemma}[theorem]{Lemma}
\newtheorem{corollary}[theorem]{Corollary}

\theoremstyle{definition}
\newtheorem{definition}[theorem]{Definition}

\theoremstyle{remark}

\newcommand{\field}[1]{\mathbb{#1}}
\newcommand{\R}{\field{R}}
\newcommand{\N}{\field{N}}
\newcommand{\C}{\field{C}}
\newcommand{\Z}{\field{Z}}
\newcommand{\Q}{\field{Q}}

\begin{document}
 \title[Cohomological equation for geodesic flows] 
 {Cohomological equation for geodesic flows on flat surfaces}

\author{Giovanni Forni}
\email{gforni@umd.edu}
\author{Nelson Moll}
\email{nmoll@umd.edu}

\address{Department of Mathematics, William E. Kirwan Hall, 4176 Campus Dr, University of Maryland, College Park, MD 20742, USA.}

\begin{abstract} We prove the existence of solutions of the cohomological equation for the geodesic flow on the unit tangent bundle of a compact flat surface with finitely many cone points. We also prove the ergodicity of the holonomy foliation for surfaces with non-rational holonomy, and the cohomology-free property of the
horizontal foliated Laplacian under a simultaneous Diophantine condition.

\end{abstract}

\maketitle 
\tableofcontents 

\section{Introduction}
In this paper, we will prove the existence of solutions of the cohomological equation for the geodesic flow on the unit tangent bundle of a compact flat surface with finitely many cone points.  Our approach is based on harmonic analysis, with the aim of adapting 
\cite{Forni1997} and \cite{Forni2022Twisted} from the case of translation surfaces to that of flat surfaces with non-rational holonomy.  We prove, under an irrationality condition on the cone angles, that the horizontal (holonomy) foliation is ergodic and derive by a Cheeger-type argument bounds on the eigenvalues  of the horizontal foliated Laplacian.  These bounds are proved under a simultaneous Diophantine condition on the angles of the cone points.  As an immediate consequence we prove that the foliated Laplacian is hypoelliptic with range of codimension one. These results will allow us to find solutions of the cohomological equation for the geodesic flow on the surface. Distributional solutions exist for all sufficiently smooth data of zero average, while smooth solutions exist for data in the kernel (perpendicular) of the infinite dimensional space of invariant distributions. In other terms,
the flat geodesic flow is stable in the sense of A.~Katok
\cite{Katok99}, \cite{Katok03}.

\smallskip

Let the pair $(S,R)$ be a flat surface together with a flat metric $R$ that has finitely many cone points at the set $\Sigma=\{p_1, \dots, p_\sigma\} \subset S$.  We will consider the geodesic flow of $R$, which is euclidean away from each cone point.  
The unit tangent bundle $M := T_1(S \setminus \Sigma)$ over the complement of the set of cone points has an infinitesimal homogeneous structure
given by vector fields $\{X, Y, \Theta\}$ satisfying the commutation relations
$$
[X,Y]=0 \,, \quad [\Theta, X]=Y \,, \quad  [\Theta, Y]= -X\,.
$$
The flow $\phi_\R^X$ is the geodesic flow of the flat metric $R$, which acts on the tangent bundle $TS$ by the formula $\phi_t^X(v) = \gamma_v(t)$, where $\gamma_v(t)$ is the unit tangent vector to the geodesic with initial condition equal to $v\in TS$.  The flow $\phi_\R^Y$ denotes the orthogonal geodesic flow,
which acts on the tangent bundle $TS$ by the formula $\phi_t^X(v)^\perp = \gamma_{v^\perp}(t)$, where $\gamma_{v^\perp}(t)$ is the unit tangent vector to the geodesic with initial condition equal to the perpendicular vector $v^\perp$ to $v\in TS$ in the positive direction. Because of the presence of cone points, the flows $\phi_\R^X$ and $\phi_\R^Y$ are defined on the  complement of a countable union of hypersurfaces (given by all unit tangent vectors generating geodesics which end in cone points). 
The flow $\phi^\Theta_\R$ is the flow of rotations in the fiber of $T_1(S\setminus \Sigma)$ which is everywhere defined and smooth on $T_1(S\setminus \Sigma)$.

Let $\{\hat X, \hat Y, \hat \Theta\}$ the system of $1$-forms to 
$\{X, Y, \Theta\}$. We have the structural equations
$$
d \hat X= - \hat \Theta \wedge \hat Y, \quad d\hat Y = \hat \Theta \wedge \hat X
\quad \text{ and }\quad d \hat \Theta =0 \,.
$$
Let $\text{vol}_M = \hat X \wedge \hat Y \wedge \hat \Theta$ denote the volume
form on $M:=T_1(S\setminus\Sigma)$. Since by the above formulas
$$
(d \circ\imath_{X}) \text{vol}_M = (d \circ\imath_{Y}) \text{vol}_M = (d \circ\imath_{\Theta})\text{vol}_M=0 \,,
$$
it follows that the flows $\phi^X_\R$, $\phi^Y_\R$ and $\phi^\Theta_\R$ are
volume preserving, that is, in terms of the Lie derivatives with respect to the generators, 
$$
\mathcal L_X \text{vol}_M = \mathcal L_Y \text{vol}_M = \mathcal L_\Theta \text{vol}_M =0\,.
$$

We will be concerned with the dynamics and geometry of horizontal (holonomy)
foliation, which we now introduce.

\begin{definition}
\label{def:hol_fol}
The foliation $\mathcal{F}_H$ in $M$ tangent to the $2$-dimensional integrable distribution $\{X,Y\}$ (which exists by the Frobenius theorem) is called the
\emph{horizontal} or \emph{holonomy} foliation. Its leaves coincide
with the equivalence classes of unit tangent vectors under the equivalence relation induced by the parallel transport with respect to the flat metric.
\end{definition}
The leaves of the horizontal foliation are translation surfaces, which are
compact if $(S,R)$ is itself a translation surface, and non-compact if $(S,R)$
has non-rational holonomy. In the latter case, the leaves have the topology of
the Loch Ness monster, as proved in \cite{Valdez_2009b}. 

Our first result is the following

\begin{theorem}
\label{thm:hol_fol_erg}
Suppose that $(R,S)$ has non-rational holonomy (a condition that holds in particular if at least one of the cone angles is irrational).  Then the horizontal foliation $\mathcal F_H$ is ergodic, that is, any square integrable function on $M$ constant along almost every leaf of $\mathcal F_H$ is almost everywhere constant.
\end{theorem}

Following the work of F.~Valdez \cite{VALDEZ_2009a}, Theorem 
\ref{thm:hol_fol_erg} has an application to the dynamics of
a class of holomorphic homogeneous foliations on $\C^2$.
In fact, the work \cite{VALDEZ_2009a} introduces a dictionary
between holomorphic homogeneous foliations on $\C^2$ and billiards in polygons. In particular, let $\mathcal F$ 
 denote the generic leaf of a homogeneous foliation  of $\C^2$ within the class considered in \cite{VALDEZ_2009a}
(see formulas $(1)$ for the foliations corresponding to triangles and $(8)$ for those corresponding to general 
polygons). Let $\mathcal G$ the projection of $\mathcal F$
onto the projective space $\mathbb P^3(\R)$ under the natural
projection $\C^2\setminus \{0\} \equiv \R^4\setminus \{0\} \to 
\mathbb P^3(\R)$. It is proved in \cite[Prop.3]{VALDEZ_2009a}
that every non-singular leaf of $\mathcal G$ is either closed
or dense (in fact, it is dense if at least one of the parameters of the foliation $\mathcal F$ is irrational). As a consequence
of Theorem \ref{thm:hol_fol_erg} we immediately have

\begin{corollary} Let $\mathcal F$ be a holomorphic homogeneous foliation on $\C^2$ in the class considered in \cite{VALDEZ_2009a} and let $\mathcal G$ its projection on
$\mathbb P^3(\R)$. The generic leaf of the foliation 
$\mathcal G$ is either closed (in the completely rational case) or equidistributed (in the non-rational case).
\end{corollary}

\bigskip
We introduce a simultaneous Diophantine condition on the holonomy of a flat surface. Let $g_1, \dots, g_d \in \pi_1(S\setminus \Sigma)$ be generators of
the fundamental group of the punctured surface $S\setminus \Sigma$. Let 
$2\pi \alpha_1, \dots, 2\pi \alpha_d$ denote the angles of rotations of the circle given by parallel transport along paths representing $g_1, \dots, g_d$
(which are well defined since the metric is flat).
In particular, if the surface $S$ has genus $0$ then the fundamental group $\pi_1(S\setminus \Sigma)$ is generated by loops around the cone points. 
Let then $2 \pi (\alpha_1+1, \dots, \alpha_\sigma+1)$ denote the angles at the cone points $\{p_1, \dots, p_\sigma\}$.  We recall that the vector $\alpha=(\alpha_1, \dots, \alpha_\sigma)$ is related to the genus $g(S)$ of the surface by the condition
$$
\sum_{i=1}^\sigma \alpha_i = 2 g(S)-2\,.
$$
We note that the vector of rotation angles of the holonomy maps corresponding to loops around the cone points is precisely 
 $2 \pi (\alpha_1+1, \dots, \alpha_\sigma+1)$.

\medskip 
We introduce the following Diophantine condition on the holonomy of a flat surface with finitely many conical singularities.

\begin{definition}  
\label{def:Dioph_cond}
The holonomy representation $\pi_1(S\setminus \Sigma) \to S^1$ of a flat surface $(S,R)$ with conical singularities at $\Sigma$ satisfies a simultaneous
Diophantine condition of exponent $\gamma \geq 1/d$ if there exists a set of generators of $\pi_1(S\setminus \Sigma)$ such that the corresponding vector $2\pi (\theta_1, \dots, \theta_d)$ of rotation angles is simultaneously Diophantine of exponent $\gamma \geq 1/d$ in the following sense. Let $d: \R^2 \to \R^+ \cup\{0\}$ denote the euclidean distance. There exists a constant $C>0$ such that, for all $n\in \Z \setminus \{0\}$,
\begin{equation}
d(n \theta_1, \mathbb{Z}) + \cdots + d(n \theta_d, \mathbb{Z}) \geq \frac{C}{n^{\gamma}} \,.
\label{sdioph}
\end{equation}
\end{definition} 

A crucial step in the analysis of the cohomological equation for translation
flows is the following cohomology-free result for the partial (foliated) 
Laplacian 
$$
H = -(X^2 + Y^2).
$$
\begin{theorem}
\label{thm:Hor_Lapl}
Suppose the holonomy of the flat surface $S$ with conical singularities at $\Sigma$ is simultaneously Diophantine of exponent $\gamma >0$.  For function $f$ on $M$ with $L^2$ bounded $\Theta$-derivatives up to order $t>\gamma+1$ and with $\int_M f d\text{vol}_M = 0$, there is an $L^2$ solution $u$ to the cohomological equation 
\[
Hu = f.
\]
\end{theorem}

We recall that the \emph{cohomological equation} for a flow on a manifold $M$ can be written in terms
of its generator $X$ as the linear PDE
$$
Xu =f \,.
$$
It is immediate that \emph{$X$-invariant distributions}, that is, distributional
solutions $D \in \mathcal E'(M)$ of the equation $X D =0$, are obstructions to the existence of smooth solutions $u$ of the cohomological equation. In fact,
if for a given smooth function $f$ on $M$, there exists a function $u$ smooth on $M$ such that $Xu=f$, then for every $X$-invariant distributions $D$ we have
$$
D(f) = D(Xu) = (XD)(u) =0 \,.
$$
The \emph{stability} of a flow (in the sense of A.~Katok \cite{Katok99}, \cite{Katok03}) means that $X$-invariant distributions are a complete set of obstructions for the existence of a smooth
solution of the cohomological equation with smooth data.

The main results of this paper establish the stability of the geodesic flow of a flat surface with simultaneously Diophantine holonomy, as well as the existence of an infinite dimensional space of distributional obstruction, given by invariant distributions. These results (see Theorems \ref{thm:CE_dist_sol}, \ref{thm:CE_smooth_sol} and Corollary \ref{cor:CE_inv_dist}) can be summarized as follows:

\begin{theorem}
\label{thm:Intro_Cohom}
For any $\gamma >0$ there exists $t_0:=t_0(\gamma)$ such that following holds.
Suppose that the holonomy of $S$ satisfies a simultaneous Diophantine condition of exponent $\gamma >0$. Then 
\begin{itemize}
\item for any function $f$ of zero average with square integrable derivatives up to order $t>t_0$, there exists a distributional solution $u$ of the cohomological equation $Xu=f$;
\item if in addition $f$ belongs to the kernel of all $X$-invariant
distributions, then there exists a solution $u$ which is square integrable with all of its derivatives up to order $t-t_0$;
\item The space of all $X$-invariant distributions of Sobolev order $t_0$ has countable
dimension.
\end{itemize}
\end{theorem}
\section{Structure of the Paper}
Our approach to proving existence of distributional 
and smooth solutions of the cohomological equation
is based on {\it a priori} estimates (see Theorem \ref{thm:aprioriest_smooth}). These in turn
are based on the harmonic analysis on flat surfaces
recalled in section \ref{sec:GA_flat}, in particular 
Lemma \ref{lemma:CR_identity}, and on a priori bounds
for the partial (horizontal) Laplacian (see Theorem~\ref{thm:H_apriori_est}). The analysis of the horizontal
Laplacian is based on lower bounds
on eigenvalues of the partial (horizontal) Laplacian,
derived from a Cheeger-type estimate (see Theorem 
\ref{thm:sdestimate}). 

\smallskip
The paper is organized as follows. In section \ref{sec:GA_flat}
we establish basic results on geometry and analysis on flat surfaces with conical points. In particular we introduce convenient Sobolev spaces (which distinguish between horizontal and vertical smoothness) in subsection \ref{ssec:Sobolev}, 
and recall the  creation/annihilation properties of foliated Cauchy-Riemann operators in subsection~\ref{ssec:CR}.

In section \ref{sec:horizontal} we prove results on the dynamics and analytical properties of the horizontal (holonomy) $2$-dimensional foliation. In subsection~\ref{ssec:hor_ergodicity}
we prove that the horizontal foliation is ergodic. 
The rest of the section is devoted to the proof of results
on the analysis of the horizontal foliated Laplacian. We prove
in subsection \ref{ssec:Cheeger} that the horizontal Laplacian is globally hypoelliptic (cohomology free). This result is based 
on Cheger-type lower bounds on its eigenvalues. 
Finally, in section \ref{sec:CE} we complete the analysis of the cohomological equation for the geodeic flow. In subsection \ref{ssec:solutions} we give a proof of existence of distributional and smooth solutions and in subsection \ref{ssec:Inv_dist} we describe the structure of the space
of distributional obtructions to the existence of smooth solutions (invariant distributions).

\section{Geometry and Analysis on Flat Surfaces}
\label{sec:GA_flat}
Let $(S,R)$ denote a flat surface $S$ with a finite set $\Sigma\subset S$ of conical singularities. 

\subsection{Local homogeneous structure}
Let $\{X, Y, \Theta\}$ denote the frame composed of the generator $X$ of the \emph{geodesic flow}, of the generator $Y$ of the \emph{orthogonal geodesic flow} and of the generator $\Theta$ of the \emph{rotations in the fibers} of $M=T_1(S\setminus \Sigma)$ with respect to the flat metric $R$ on $S$.
We recall the commutation relations
\begin{equation}
\label{eq:comm_rel_main}
[X,Y]=0\,, \quad [\Theta, X]= Y\,, \quad [\Theta, Y]=-X\,.
\end{equation}
Important auxiliary operators are the foliated \emph{Cauchy-Riemann operators}
$\partial^\pm = X \pm i Y$. By the main commutation relation we derive the following commutation identities for the Cauchy-Riemann operators
\begin{equation}
\label{eq:comm_rel_CR}
[\partial^+, \partial^-]=0\,, \quad  [\Theta, \partial^\pm] = \mp i \partial^\pm\,.
\end{equation}

\subsection{A Fourier decomposition}
Since $\Theta$ integrates to a circle action on $M$, it induces a splitting
of the space $L^2(M,\text{vol}_M)$:

For all $k\in \Z$, let  
\begin{equation}
E_k= \{f \in L^2(M, \text{vol}_M) : f\circ \phi^\Theta_\theta  =  e^{i k \theta} f \,, \text{ for all } \theta\in \mathbb {R} \}.  
\end{equation}
\begin{lemma} The following orthogonal splitting holds:
$$
L^2(M, \text{vol}_M) = \bigoplus_{k\in \Z} E_k\,.
$$
\label{fouriertangent}
\end{lemma}
\begin{proof}
We recall the argument for the convenience of the reader.
Let $\phi^{\Theta}_t$ be the flow given by rotation of a vector $v \in M:=T_1(S\setminus \Sigma)$.  For all $k\in \Z$, define the projection $\pi_k : L^2(M, \text{vol}_M) \to L^2(M, \text{vol}_M)$ by 
\begin{equation}
\pi_k(f)  = \frac{1}{2\pi} \int \limits_{0}^{2 \pi} e^{- i k \theta} (f \circ \phi^{\Theta}_\theta) d\theta \,.
\end{equation}
It can be verified that, for all $k\in \Z$, $\pi_k$ is the orthogonal projection onto the eigenspace $E_k$. In fact, for all $r>0$ by change of variables and periodicity, we have
$$
\begin{aligned}
\pi_k(f)\circ \phi^\Theta_r &= \frac{1}{2\pi}\int \limits_{0}^{2 \pi} e^{- i k \theta} (f \circ \phi^{\Theta}_{\theta+r}) d\theta \\ &= \frac{e^{i k r}}{2\pi}\int \limits_{r}^{2 \pi+r} e^{- i k \theta} (f \circ \phi^{\Theta}_\theta) d\theta = e^{ i k r} \pi_k(f)\,.
\end{aligned}
$$
In addition, it is immediate that, for all $k\in \Z$,
$$
\pi_k^2 (f) = \frac{1}{2\pi} \int \limits_{0}^{2 \pi} e^{-i k \theta} (\pi_k(f) \circ \phi^{\Theta}_{\theta}) d\theta = \frac{1}{2\pi} \int \limits_{0}^{2 \pi}  \pi_k(f) d\theta = \pi_k(f)\,.
$$
and that for Since for $h \not = k \in \Z$ 
$$
\pi_h \pi_k (f) = \frac{1}{2\pi} \int \limits_{0}^{2 \pi} e^{- i k \theta} (\pi_h(f) \circ \phi^{\Theta}_{\theta}) d\theta = \frac{1}{2\pi} \int \limits_{0}^{2 \pi} e^{- i (k-h) \theta} f d\theta = 0\,.
$$
Since by orthogonality
$$
\sum_{k\in \Z} \Vert \pi_k(f) \Vert^2_{L^2(M, \text{vol}_M)} \leq 
 \Vert f\Vert^2_{L^2(M, \text{vol}_M)}
$$
and by Fourier series expansion for periodic functions on the real line we have that, for all $f\in L^2(M, \text{vol}_M)$ and for almost all $(x,\xi) \in M$,
$$
f(x,\xi) = \sum_{k\in \Z} \pi_k(f) (x,\xi)\,,
$$
we conclude that $\sum \limits_{k \in \mathbb{Z}} \pi_k = \text{Id}_{L^2(M, \text{vol}_M)}$.  
\end{proof}

\subsection{Local coordinates}

The smooth structure on $M$ induced by the flat metric with conical singularities can be understood in local coordinates. 
At a regular point there exists a (locally defined) complex canonical coordinate $z=x+iy $ on $S$ such that the metric $R$ is euclidean:
$$
R = \left(dx^2 + dy^2 \right)^{\frac{1}{2}}\,.
$$
As a consequence, the unit tangent bundle is locally trivial with local coordinates $(z, \theta)$ and we have the following local formulas
$$
X = \cos \theta \frac{\partial}{\partial x} + \sin \theta \frac{\partial}{\partial y} \,, \quad Y = -\sin \theta \frac{\partial}{\partial x} + \cos \theta \frac{\partial}{\partial y} \,, \quad  \Theta= \frac{\partial}{\partial \theta}
$$
with invariant volume form $\text{vol}_M = dx \wedge dy \wedge d\theta$.

At a conical point of angle $2\pi (\alpha +1)$, there exists a (locally defined) complex canonical coordinate $z=x+iy $ on $S$ such that the metric $R$
has the expression 
\begin{equation}
\label{conepointmetric}
R = \vert z \vert^\alpha\left(dx^2 + dy^2 \right)^{\frac{1}{2}}.
\end{equation}
As a consequence, the unit tangent bundle is not locally trivial and has local coordinates $(z, \theta)$ with the following identification
\begin{equation}
\label{eq:cone_point_coord}
( e^{2\pi i} z, \theta) \equiv ( z, \theta + 2\pi \alpha)\,.
\end{equation}
With respect to the above coordinates we have the following local formulas 
$$
\partial ^+ = e^{-i\theta} \bar z^{-\alpha} \frac{\partial}{\partial {\bar z}}\,, \quad \partial ^- = e^{i\theta} z^{-\alpha} \frac{\partial}{\partial z}\,,  \quad \Theta = \frac{\partial}{\partial \theta}$$
and, as a consequence,
\begin{equation}
\label{eq:loc_cone}
\begin{aligned}
X &= \frac{ \partial ^++ \partial^-}{2} = \frac{1}{2} (e^{-i\theta} \bar z^{-\alpha} \frac{\partial}{\partial {\bar z}}+  e^{i\theta} z^{-\alpha} \frac{\partial}{\partial z}    )  \,, \\ Y &= \frac{ \partial ^+- \partial ^-}{2i} = \frac{1}{2i} (e^{-i\theta} \bar z^{-\alpha} \frac{\partial}{\partial {\bar z}}-  e^{i\theta} z^{-\alpha} \frac{\partial}{\partial z})\,, 
\end{aligned}
\end{equation}
with invariant volume form $\text{vol}_M = \vert z \vert^{2\alpha} dx \wedge dy \wedge d\theta$.

We introduce Sobolev spaces on $M$ below.  First we introduce the space
of smooth functions on the unit tangent bundle of a flat surface with 
conical points.

\begin{definition}
\label{def:smooth_functions}
The natural space $C^\infty_\Sigma (M)$ of smooth functions on the unit tangent bundle of a surface $S$ with cone singularities at $\Sigma$ is defined as follows.  A function $f \in C^\infty_\Sigma (M)$ if it restricts to a smooth function of the canonical local coordinates $(z, \theta)$ in the unit tangent bundle of a neighborhood of any regular point and if it restricts to a smooth
function of the variables 
$$
(z, \bar{z}, e^{-i\theta} z^{\alpha}, e^{i\theta} \bar {z}^{\alpha})
$$
in the unit tangent bundle of a neighborhood of any conical point of angle
$2\pi (\alpha+1)$.
\end{definition}
The above definition is motivated by the condition that the space
$C^\infty_\Sigma(M)$ be the common domain of all iterates of the operators
$\partial^+$, $\partial^-$, $\Theta$ or equivalently $X$, $Y$, $\Theta$.

\subsection{Sobolev spaces}
\label{ssec:Sobolev}

For $(r, s) \in \mathbb{N} \times \mathbb{N}$, we define the norm 
\begin{equation}
\vert f \vert_{r,s} = \Big(\sum \limits_{i + j \leq r} \sum \limits_{l \leq s}\|X^iY^j \Theta^l f \|_{L^2(M, \text{vol}_M)}^2\Big)^{1/2}.
\label{sobnorm}
\end{equation}
We then define the $L^2$ Sobolev space $H^{r,s}(M)$ on $M$ as the completion with respect to the norm $\vert\cdot\vert_{r,s}$ of the space $C_{\Sigma}^{\infty}(M)$ of smooth functions. We also define $H^{-r,-s}(M)$ to be the dual space to $H^{r,s}(M)$. 

We note that  for all $r,s \in \Z$ there exists $C_{r,s} >0$ such that, for all  $f \in H^{r, s}(M)$ we have 
\begin{equation}
C_{r,s}^{-2} \vert f \vert_{r,s}^2\, \leq \,    \sum \limits_{n=0}^{\infty} (1+n^{2s}) \vert \pi_n(f) \vert_{r,0}^2  \,  \leq \,  C_{r,s}^2 \vert f \vert_{r,s}^2
\end{equation}

The Lie derivative operators $X$ and $Y$, as well as the Cauchy-Riemann
operators $\partial^\pm$ are defined on the common domain $H^{1,0}(M)$. 
Since the flows $\phi^X_{\R}$ and $\phi^Y_{\R}$ are
volume preserving, the Lie derivative operators $X$ and $Y$ are skew-symmetric on 
$H^{1,0}(M)$. Since the set of the  singular trajectories of $\phi^X_{\R}$ and $\phi^Y_{\R}$ are countable unions
of codimension one submanifolds and have therefore zero volume, the operators
$X$ and $Y$ are essentially skew-adjoint \cite{Ne59}. As for the Cauchy-Riemann operators,
it follows that the adjoint of $\partial^\pm$ on the common domain $H^{1,0}(M)$ is an extension of $-\partial^ \mp$ on the same domain. 

The Cauchy-Riemann operators satisfy the following crucial
identity.

\begin{lemma} 
\label{lemma:CR_identity}
For all $v\in H^{1,0}(M)$ we have
$$
\Vert \partial^+ v \Vert^2_{L^2(M, \text{vol}_M)} = 
\Vert \partial^- v \Vert^2_{L^2(M, \text{vol}_M)} = 
 \Vert X v \Vert^2_{L^2(M, \text{vol}_M)} + \Vert Y v \Vert^2_{L^2(M, \text{vol}_M)}\,.
$$
\end{lemma}
\begin{proof} By the definition of the space $H^{1,0}(M)$, it
suffices to prove the statement for smooth functions $v\in C^\infty_\Sigma(M)$ and conclude by a density argument. 

For $v\in C^\infty_\Sigma(M)$, we have the commutation identity $XY v= YXv$, hence
$$
\langle Yv,Xv \rangle = -\langle XYv,v \rangle = -\langle YXv,v \rangle = \langle Xv,Yv \rangle\,.
$$
It then follows that
$$
\begin{aligned}
\Vert \partial^\pm v \Vert^2_{L^2(M, \text{vol}_M)} &=
\langle (X\pm iY)v, (X\pm iY)v \rangle \\ &= \langle Xv,Xv \rangle +
\langle Yv,Yv \rangle \pm i (\langle Yv,Xv \rangle -\langle Xv,Yv \rangle) \\ &= \Vert X v \Vert^2_{L^2(M, \text{vol}_M)} + \Vert Y v \Vert^2_{L^2(M, \text{vol}_M)} \,,
\end{aligned}
$$
as stated.
\end{proof}

Since the flow $\phi^\Theta_{\R}$ leaves invariant the plane generated by
$\{X,Y\}$, we derive the following result on the orthogonal projections onto the eigenspaces of $\phi^\Theta_{\R}$.

\begin{lemma} 
\label{lemma:pi_n}
For all $r \in \N$ and for each $n\in \Z$ the image of the 
Sobolev space $H^{r,0}(M)$ under the projection operator $\pi_n: L^2(M, \text{\rm vol}_M) \to E_n$ is contained in the space $H^{r,0}(M) \cap E_n$ and
we have the following estimates. For every $r\in \N$ there exists a constant $C_r>0$ such that for all $f \in H^{r,0}(M)$ we have
$$
\vert \pi_n(f) \vert_{r,0} \leq C_r \vert f \vert_{r,0} \,.
$$
\end{lemma}
\begin{proof}
We recall that 
$$
\pi_n(f) = \frac{1}{2\pi} \int_0^{2\pi}  e^{-in t}  f \circ \phi^{\Theta}_t dt\,.
$$
Since $(\phi^\Theta_t)_\ast (X) = (\cos t) X + (\sin t) Y$ and $(\phi^\Theta_t)_\ast (Y) = (-\sin t) X + (\cos t) Y$ by the commutation relations, it follows that (in the weak sense)
$$
\begin{aligned}
X \pi_n(f) = \frac{1}{2\pi} \int_0^{2\pi}  e^{-in t}  [(\cos t) X + (\sin t) Y]f \circ \phi^{\Theta}_t dt\,, \\
Y \pi_n(f) = \frac{1}{2\pi} \int_0^{2\pi}  e^{-in t}  [(-\sin t) X + (\cos t) Y]f \circ \phi^{\Theta}_t dt\,.
\end{aligned}
$$
It follows that for $f \in H^{1,0}(M)$ we have that $\pi_n(f) \in H^{1,0}(M)$, for all $n\in \Z$, and there exists a constant $C_1>0$ such that, for all $f\in H^{1,0}(M)$, we have
$$
\vert  \pi_n(f) \vert_{1,0} \leq C_1 \vert f \vert_{1,0}\,.
$$
The statement can be proved similarly, and we leave the argument to the reader.
\end{proof}

\subsection{Creation and annihilation}
\label{ssec:CR}
By the commutation relations \eqref{eq:comm_rel_CR}, the Cauchy-Riemann
operators are operators of creation and annihilation for the spectral
decomposition of the rotation. In fact we have 

\begin{lemma}
\label{lemma:plus_minus}
For each $n\in \mathbb{Z}$, the Cauchy-Riemann operators $\partial^\pm : E_n \cap H^{1,0}(M) \to E_{n-1}$ and $\partial^- : E_n \cap H^{1,0}(M) \to E_{n+1}$ are well-defined and bounded. 
\end{lemma}
\begin{proof}
Let $f \in E_n \cap H^{1,0}(M)$ then $\partial^\pm f \in L^2(M, \text{vol}_M)$ and by definition of the Sobolev spaces 
$$
\Vert \partial^\pm f \Vert_{L^2(M, \text{vol}_M)} \leq \vert f \vert_{1,0}\,.
$$
By the commutation relations we get (taking the derivative with respect to $\Theta$ in weak sense)
$$
\Theta(\partial^\pm f) =   [\Theta, \partial^\pm] f + \partial^\pm \Theta f \\
 =  \mp i \partial^\pm f  + in \partial^\pm \Theta f  = i (n\mp 1) \partial^\pm f \,,
$$
hence $\partial^\pm f\in E_{n\mp 1}$, as claimed.
\end{proof}

Let $\mathcal M^\pm$ denote the kernels of $\partial^\pm$ in
$L^2(M, \text{vol}_M)$, or in other terms the kernels of the
adjoint operators $(\partial^\mp)^\ast$ (which are extensions
of $-\partial^\pm$ respectively).

The space $\mathcal M^+$ is the space of (foliated) meromorphic
functions and similarly $\mathcal M^-$ is the space of (foliated) anti-meromorphic functions. We have the well-known descriptions:
\begin{lemma} The space $\mathcal M^\pm$ decomposes as as orthogonal sum of its components in the eigenspaces of the rotation:
$$
\mathcal M^\pm = \bigoplus_{n\in \Z} (\mathcal M^\pm \cap E_n). $$
For every $n\in \Z$, the space $\mathcal M^+_n:= \mathcal M^\pm \cap E_n$ and $\mathcal M^-_n:= \mathcal M^-_n \cap E_n$ are in bijective correspondence with the space of holomorphic and, respectively, anti-holomorphic $n$-differentials. 
\end{lemma}
\begin{proof} For every $m^\pm \in \mathcal M^\pm$ we have
the decomposition
$$
m^\pm = \sum_{n\in \Z} \pi_n (m^\pm)\,.
$$
By Lemma \ref{lemma:plus_minus} it follows that $\pi_n (m^\pm)
\in \mathcal M^\pm$, since for all $v\in H^{1,0}(M)$,
$$
\langle \pi_n (m^\pm), \partial^\mp v\rangle = \langle m^\pm,  \pi_n (\partial^\mp v)\rangle =
\langle m^\pm, \partial^\mp ( \pi_{n\pm 1} (v) )\rangle =0\,.
$$
For every $n \in \Z$, let now $\mathcal D^+_n$ and $\mathcal D^-_n$ denote the space of holomorphic and, respectively, anti-holomorphic $n$-differentials, defined
as holomorphic and, respectively, anti-holomorphic sections
of the $n$-th power of the canonical bundle of the Riemann surface $S$. There exist maps $\imath_n^\pm: \mathcal D^\pm_n \to \mathcal M^\pm_n$ defined as follows: for all $d^\pm_n \in \mathcal D^\pm_n $,
$$
\imath_n^\pm (d^\pm_n) (x,\xi) = d^\pm_n(x) (\xi)\,, \quad \text{for all } (x,\xi) \in T_1(S) \,.
$$
It can be verified by a local calculation that,since the
differentials $d^+_n$ and $d^-_n$ are respectively holomorphic
and anti-holomorphic, the functions $\imath_n^+ (d^+_n)$ and
$\imath_n^- (d^-_n)$ belong to the space $\mathcal M^+$ and
$\mathcal M^-$ respectively. In addition, since $d^\pm_n$ are
(smooth) $n$-differentials, the functions $\imath_n^\pm (d^\pm_n)$ are square integrable and belong to the space
$E_n$. Finally, it can be verified that given any function
$m^\pm_n \in \mathcal M^\pm_n$, there exists a unique $n$-differentials $d^\pm_n$ such that
$$
 d^\pm_n(x) (\xi) =  m^\pm_n (x,\xi) \,, \quad \text{for all } (x,\xi) \in T_1(S) \,,
$$
hence the maps $\imath_n^\pm: \mathcal D^\pm_n \to \mathcal M^\pm_n$ are isomorphisms.
\end{proof}

\section{The horizontal (holonomy) foliation}
\label{sec:horizontal}

In this section we examine the dynamical properties of the horizontal
foliation, including its ergodicity for surfaces with non-rational holonomy
and the cohomological equation for the foliated Laplacian under Diophantine
conditions on the holonomy representation. Bounds on the inverse of the 
foliated Laplacian will be relevant in the study of the cohomological equation
for the flat geodesic flow.

\subsection{Ergodicity}
\label{ssec:hor_ergodicity}
We recall that according to Definition \ref{def:hol_fol} the horizontal (holonomy) foliation $\mathcal F_H$ has leaves tangent to the integrable distribution $\{X,Y\} \subset TM$.  We have the following result.

\begin{theorem}
\label{thm:holonomy_erg}
Suppose that the holonomy representation $\pi_1(S\setminus \Sigma)$ is non-rational.  Then any function $u \in L^2(M, \text{vol}_M)$ constant along the leaves of the horizontal foliation $\mathcal F_H$  is constant, hence the foliation
$\mathcal F_H$ is ergodic. 
\end{theorem}
\begin{proof}
Let $u\in L^2(M, \text{vol}_M)$ be any function constant along the leaves of
the foliation $\mathcal F_H$. Since each leaf of $\mathcal F_H$ is invariant under the flows $\phi^X_{\R}$ and $\phi^Y_{\R}$ it follows that 
in the weak $L^2$ sense we have
$$
Xu=Yu= 0\,,  \quad \text{ hence } \quad \partial^\pm u =0 \,,
$$
Now expand $u = \sum \limits_{n} u_n$ in terms of the eigenspaces of $\phi^\Theta_{\R}$. We claim that the function $u_n$ is constant
along the leaves of $\mathcal F_H$, for all $n \in \Z$. 
In fact, for each $n\in \Z$ and for any $v \in H^{1,0}(M)$, since
$\partial ^\pm : E_n \cap H^{1,0}(M) \to E_{n\mp 1} $ by Lemma \ref{lemma:plus_minus} and $\pi_{n\pm 1} (v) \in H^{1,0}(M)$ by Lemma \ref{lemma:pi_n}, we have 
$$
\langle u_n, \partial^\pm v \rangle = \langle \pi_n(u), \partial^\pm v \rangle
= \langle u, \pi_n (\partial^\pm v) \rangle = \langle u, \partial^\pm \pi_{n\mp 1}(v) \rangle =0\,.
$$
We finally claim that under the hypotheses that the holonomy representation
is non-rational, then $u_0 \in \C$ is constant and $u_n =0$ for all
$n \in \Z \setminus \{0\}$.

In fact, since the holonomy representation is not rational, for any point
$x\in S\setminus \Sigma$ there exists a closed path $\gamma: [0,1] \to S\setminus \Sigma$ such that $\gamma (0)=\gamma(1) =x$ and such that the
holonomy map corresponding to $\gamma$ (on the fiber $M_x$) is an irrational
rotation. For any $(x,\xi) \in M_x$, let then $\hat \gamma_\xi:[0,1] \to M$ denote the lift of $\gamma$ defined by the parallel transport of $\xi\in M_x$ along
$\gamma$. There exists $\theta \in \R$ such that $2\pi \theta \not \in \Q$ such that
$$
\hat \gamma_\xi (1) = \phi^\Theta_{\theta} ( \hat \gamma_\xi(0) )= 
\phi^\Theta_{\theta} (x,\xi)\,.
$$
For all $n\in \Z$, since $u_n$ is constant along the leaves of the
foliation $\mathcal F_H$, hence constant along the path $\hat \gamma_\xi$ and
since $u_n \in E_n$ we have
$$
e^{2\pi i n \theta} u_n(x,\xi)=  u_n (\hat \gamma_\xi (1)) = u_n  (\hat \gamma_\xi (0)) = u_n(x,\xi)\,.
$$
For all $n\not =0$, since $2\pi \theta$ is irrational, it follows that 
$u_n(x,\xi)=0$, and as a consequence, since $(x,\xi)$ was any given point in $M$,
we conclude that $u_n \equiv 0$. 

For $n=0$, we have that $u_0$, as a $\phi^\Theta_{\R}$-invariant function
is the pull-back of a function on~$S$. Since $u_0$ is constant along leaves of
$\mathcal F_H$ which project onto $S$ by local diffeomorphisms, it follows that
$u_0$ is locally constant, hence constant. 
\end{proof}

As a consequence of the ergodicity of the holonomy foliation we
prove that the flat geodesic flow and its rotations have no non-constant regular invariant function.

\begin{corollary} 
\label{cor:no_smooth_inv}
Suppose that the holonomy representation $\pi_1(S\setminus \Sigma)$ is non-rational.
Let $v\in H^{1,0}(M)$ be any function such that
there exists $\theta \in S^1$ such that $X_\theta v=0$. Then $v$
is constant along the holonomy foliation, hence it is constant.
\end{corollary}
\begin{proof}  We have
$$
X_\theta v= e^{-i\theta} \partial^+v + e^{i\theta}\partial^- v=0\,.
$$
Following Lemma \ref{lemma:pi_n} let us write $v =\sum_{n\in \Z} \pi_n (v)$ as a series of its projections $\pi_n(v)$ onto $E_n \cap H^{1,0}(M)$. By Lemma \ref{lemma:plus_minus} and by the above identity $X_\theta v=0$ is equivalent to the sequence
of identities
$$
e^{-i\theta} \partial^+v_{n+1} = - e^{i\theta}\partial^- v_{n-1}\,,  \quad  \text{for all } n\in \Z\,.
$$
By Lemma \ref{lemma:CR_identity} it follows that, for all $n\in \Z$, we have
\begin{equation}
\label{eq:CR_recursion}
\begin{aligned}
\Vert \partial^+v_{n+1} \Vert_{L^2(M, \text{vol}_M)} =
\Vert \partial^-v_{n-1} \Vert_{L^2(M, \text{vol}_M)} =
\Vert \partial^+v_{n-1} \Vert_{L^2(M, \text{vol}_M)} \,, \\
\Vert \partial^-v_{n-1} \Vert_{L^2(M, \text{vol}_M)} =
\Vert \partial^+v_{n+1} \Vert_{L^2(M, \text{vol}_M)} =
\Vert \partial^-v_{n+1} \Vert_{L^2(M, \text{vol}_M)}
\end{aligned}
\end{equation}
However, since $\partial^+v$ and $\partial^-v\in L^2(M, \text{vol}_M)$, we have that 
$$
\sum_{n\in \Z} \Vert \partial^+v_{n} \Vert^2_{L^2(M, \text{vol}_M)} = \sum_{n\in \Z} \Vert \partial^-v_{n} \Vert^2_{L^2(M, \text{vol}_M)} \leq \vert v \vert^2_{1,0} <+\infty \,,
$$
which by the recursion in formula \eqref{eq:CR_recursion} implies
that 
$$
\partial^+v_{n} = \partial^-v_n =0 \,, \quad \text{ for all } n\in \Z\,.
$$
We have thus prove that $\partial^+ v = \partial^-v=0$, or,
in other terms $Xv = Yv=0$, that is, $v$ is invariant along
the holonomy foliation, hence by Theorem \ref{thm:holonomy_erg}  it is constant.
\end{proof}

\subsection{A Cheeger-type bound}
\label{ssec:Cheeger}
In this subsection we prove a lower bound on the eigenvalues of the Dirichlet quadratic form of the foliated Laplacian $H=- (X^2+Y^2)$. 

We note that since the operator $H$ commutes with the flow $\phi^\Theta_{\R}$ it also commutes with the Laplacian $\Delta=-(H + \Theta^2)$. It follows that all these operators can be simultaneously diagonalized with eigenfunctions in the Sobolev space $H^{1,\infty} (M)$. 

For every $n\in \Z$, let $\lambda_n(H)$ be the least eigenvalue of the Dirichlet quadratic form of $H$ on $E_n$. We have the following

\begin{theorem}
\label{thm:sdestimate}
If the holonomy of the flat surface $(S,R)$ satisfies a simultaneous Diophantine  condition of exponents $\gamma >0$ (see Definition \ref{def:Dioph_cond}), then there is a constant $C > 0$  such that, for all
$n \in \mathbb{Z} \setminus \{0\}$,
\[
\lambda_n(H) \geq \frac{C}{n^{2\gamma}}.
\label{sdestimate}
\]
\end{theorem}

The proof of this result is based on Cheeger's proof of a lower bound for the least eigenvalue of the Laplace-Beltrami operator. Since $H$ is not the Laplacian of a Riemannian metric, we consider the family of elliptic operators 
$$
\Delta_\epsilon = H - \epsilon^2 \Theta^2\,. 
$$
For all $\epsilon>0$ the operators $\Delta_\epsilon$ are the Laplace-Beltrami operator for a metric on $M$ with conical singularities along finite many circles.  

For every $\epsilon >0$ and $n\in \Z$, since the Laplace operators $\Delta_\epsilon$ and $H$ commute (and commute with the Lie derivative operator $\Theta$) we have that
\begin{equation}
\label{eq:eigen_limit}
\lambda_n(H) = \lim_{\epsilon \to 0}  \lambda_n(\epsilon)\,.
\end{equation}
We will prove below a Cheeger-type bound on the smallest eigenvalue of the
Dirichlet for of $\Delta_\epsilon$ on the subspace $E_n$, for all $n\in \Z\setminus \{0\}$. 

Let $\mathcal S_n(M)$ denote the set of all separating hypersurfaces in 
$M$ which are $n$-covers of $S\setminus \Sigma$ under the canonical projection
of $M=T_1(S\setminus \Sigma)$ onto $S\setminus \Sigma$ (in the sense that the 
restriction of the projection to the hypersurface is onto with fibers of cardinality equal to $n\not=0$). 

Let $\text{Area}_\epsilon$ and $\text{Vol}_\epsilon$ denote the $2$-dimensional
volume (area) and  the $3$-dimensional volume of the metric $g_\epsilon$ on $M$ (defined by the condition that $\{X, Y, \epsilon \Theta\}$ is a orthonormal frame of $TM$) whose Laplace-Beltrami operator is $\Delta_\epsilon$.

For every $\mathcal S \in \mathcal S_n(M)$, let $M'_{\mathcal S}$ and  $M''_{\mathcal S}$ the components of $M\setminus {\mathcal S}$. Let then
 \begin{equation}
h_n(\epsilon) :=  \inf_{{\mathcal S} \in \mathcal S_n(M)}
 \dfrac{\text{Area}_{\epsilon}({\mathcal S})}{\min\{\text{Vol}_{\epsilon}(M'_{\mathcal S}), \text{Vol}_{\epsilon}(M''_{\mathcal S}) \}}.
 \label{h_nepsilon}
\end{equation}
We prove below the following lower bound:

\begin{theorem}
\label{thm:Cheeger}
For all $n\in \Z\setminus \{0\}$, the least Friederichs eigenvalue $\lambda_n(\epsilon)$ of the Dirichlet form of the operator $\Delta_{\epsilon}$ on $M$ restricted to $E_n$  is bounded below as
\[
\lambda_n(\epsilon) \geq \left(\frac{h_n(\epsilon)}{2} \right)^2.
\]
\end{theorem}

\begin{proof} The proof follows Cheeger's argument in \cite{Che70}. 
Let $\nabla_\epsilon$ denote the gradient of the Riemannian metric 
$g_\epsilon$ on $M$ and let $\text{vol}_\epsilon$ denote its volume element.  By the variational principle
$$
\lambda_n(\epsilon) = \inf_{f \in C^\infty_\Sigma(M, \R) \cap (E_n\oplus E_{-n})}
\frac{\Vert \nabla_\epsilon f \Vert^2_{L^2(M,\text{vol}_\epsilon)} } {\Vert f \Vert^2_{L^2(M,\text{vol}_\epsilon)}} \,.
$$
Since $f \in E_n$ with $n\not= 0$, then $f$ has zero average on $M$. For
all $\lambda \in \R$, let 
$$
\begin{aligned}\Omega^+_f(\lambda)&=\{(x,\xi)\in M \vert 
f(x,\xi) >\lambda\}, \\  \Omega^-_f(\lambda)&=\{(x,\xi)\in M \vert 
f(x,\xi) <\lambda\} \,, \\ \mathcal L_f(\lambda) &= \{(x,\xi)\in M \vert 
f(x,\xi) =\lambda\}\,.
\end{aligned}
$$ 
We note that since $f\in E_n\oplus E_{-n}$ (and it is real-valued) all of its level $\mathcal L_f(\lambda)$ in $M$ set are $n$-covers of the surface $S\setminus \Sigma$ and by the Sard theorem they are $2$-dimensional manifolds for almost all $\lambda \in \R$.

Since $M$ is connected and $f$ has zero average, there exists $\lambda_0$
such that $\text{vol}_\epsilon (\Omega^+_f(\lambda_0)) = \text{vol}_\epsilon (\Omega^-_f(\lambda_0) ) $ and by the co-area formula we have
$$
\begin{aligned}
&\Vert \nabla_\epsilon (f - \lambda_0)^2 \Vert_{L^2(\Omega^+_f(\lambda_0), \text{vol}_\epsilon)} = 2\int_{\lambda_0}^\infty A_\epsilon (\mathcal L_f(\lambda)) (\lambda-\lambda_0) d\lambda \\  & \quad \quad \geq 2 \inf_{\lambda\geq \lambda_0} 
\frac{A_\epsilon (\mathcal L_f(\lambda)) }{ \text{vol}_\epsilon (\Omega^+_f(\lambda)) } \int_{\lambda_0}^\infty \text{vol}_\epsilon (\mathcal L_f(\lambda)) (\lambda-\lambda_0) d\lambda \\
& \quad \quad\geq h_n(\epsilon) \Vert f - \lambda_0 \Vert^2_{L^2(\Omega^+_f(\lambda_0), \text{vol}_\epsilon)}
\end{aligned} $$
and, similarly,
$$
\begin{aligned}
&\Vert \nabla_\epsilon (f - \lambda_0)^2 \Vert_{L^2(\Omega^-_f(\lambda_0), \text{vol}_\epsilon)} = 2\int_{-\infty}^{\lambda_0} A_\epsilon (\mathcal L_f(\lambda)) (\lambda_0-\lambda) d\lambda \\  & \quad \quad\geq 2 \inf_{\lambda\geq \lambda_0} 
\frac{A_\epsilon (\mathcal L_f(\lambda)) }{ \text{vol}_\epsilon (\Omega^-_f(\lambda)) } \int_{\lambda_0}^\infty \text{vol}_\epsilon (\mathcal L_f(\lambda)) (\lambda_0-\lambda) d\lambda \\
& \quad \quad\geq h_n(\epsilon) \Vert f - \lambda_0 \Vert^2_{L^2(\Omega^-_f(\lambda_0), \text{vol}_\epsilon)} \,.
\end{aligned} $$
By the above abounds and by the H\"older inequality it follows that
$$
\begin{aligned}
2 \Vert \nabla_\epsilon f  \Vert_{L^2(M, \text{vol}_\epsilon)}
\Vert f -\lambda_0 \Vert_{L^2(M, \text{vol}_\epsilon)} &\geq  \Vert \nabla_\epsilon (f -\lambda_0)^2 \Vert_{L^2(M, \text{vol}_\epsilon)}
\\ &\geq h_n(\epsilon) \Vert f -\lambda_0  \Vert^2_{L^2(M, \text{vol}_\epsilon)}
\end{aligned}
$$
hence finally, since $f$ has zero average,
$$
\Vert \nabla_\epsilon f  \Vert_{L^2(M, \text{vol}_\epsilon)}^2 \geq 
\frac{h_n(\epsilon)^2}{4} \Vert f -\lambda_0 \Vert_{L^2(M, \text{vol}_\epsilon)}^2 \geq \frac{h_n(\epsilon)^2}{4} \Vert f \Vert_{L^2(M, \text{vol}_\epsilon)}^2
$$
\end{proof}

We now need to bound $h_n(\epsilon)$ from below.  Let $ 
\{\gamma_1, \dots \gamma_d \}$ denote a basis of $\pi_1(S\setminus \Sigma)$, and for all $i\in \{1, \dots, d\}$, let $2\pi \theta_i$ denote the angle of the rotation of the holonomy along $\gamma_i$.

\begin{theorem}
\label{thm:Cheeger_const}
There exists $c>0$ such that, for all $n\in \N$ and for all $\epsilon >0$ the Cheeger-type constant $h_n(\epsilon)$ 
satisfies the lower bound
$$
h_n(\epsilon) \geq  c \sum_{i=1}^d d_{\R} (n\theta_i, \Z)\,.
$$
\end{theorem}
\begin{proof}
The area form of the metric $g_\epsilon$, defined by the orthonormality condition on $\{X, Y, \epsilon \Theta\}$, is 
\begin{equation}
 dA_\epsilon :=  \Big(\frac{1}{\epsilon^2}\vert\hat Y \wedge \hat \Theta\vert^2 + \frac{1}{\epsilon^2}\vert\hat X \wedge \hat \Theta\vert^2 + \vert\hat X \wedge \hat Y\vert^2 \Big )^{1/2}.
\end{equation}
and its volume is $dV_\epsilon = \epsilon ^{-1} dV$, hence for all $\epsilon >0$ we have
$$
\begin{aligned}
h_n(\epsilon) &\geq \inf_{\mathcal S \in \mathcal S_n(M)}
 \dfrac{\int_{\mathcal S} \Big(\vert\hat Y \wedge \hat \Theta\vert^2 + 
 \vert\hat X \wedge \hat \Theta\vert^2 \Big)^{1/2} }{\min\{\text{Vol}(M'_{\mathcal S}), \text{Vol}(M''_{\mathcal S}) \}} \\ &\geq 
\inf_{\mathcal S \in \mathcal S_n(M)}
 \int_{\mathcal S} \Big(\vert\hat Y \wedge \hat \Theta\vert^2 + \vert\hat X \wedge \hat \Theta\vert^2 \Big)^{1/2}  .
\end{aligned}
$$

For each $i\in \{1, \dots, d\}$ we consider a topological annulus $A_i$ in $S\setminus \Sigma$ which is a tubular neighborhood of the loop $\gamma_i$. Let us assume that the holonomy along $\gamma_i$ is a rotation with angle an irrational multiple of $2\pi$. Then for every surface $\mathcal S \in \mathcal S_n(M)$, the inverse image of $A_i$ in $\mathcal S$ (under the restriction to $\mathcal S$ of the canonical 
projection) is an annulus $\hat A_i \subset \mathcal S$, which is an $n$-cover of the annulus $A_i$. Let $\hat \gamma_i$ a core curve of the annulus $\hat A_i$ which projects onto $\gamma_i$.
Let $\hat I_i$ denote an arc transverse to $\hat \gamma_i$
such that $\hat A_i \setminus \hat I_i$ is simply connected.

The boundary of $\hat A_i \setminus \hat I_i$ is the union 
of $(\partial \hat A_i) \setminus (\hat I_i \cap \partial \hat A_i) $ and of the arcs
$\hat I_i^\pm$ which both map to $\hat I_i$ under the inclusion
of $\hat A_i \setminus \hat I_i$ into $M$.
It can be proved that there exist a closed $1$-form $\hat V$ such that $\int_{\hat I_i} \hat V>0$ which in addition is horizontal, that is $\hat V(\Theta)=0$, and a function $\theta: \hat A_i \setminus \hat I_i \to \R$ such that $\hat \Theta = d\theta$ on $\hat A_i \setminus \hat I_i$  and $\int_{ \partial \hat A_i}  \theta \hat V =0$. 

We note that for any pair of points
$p^\pm \in \hat I^\pm_i$ such that $p^\pm$ have the same image
in $\hat I_i \subset M$ under the inclusion map, we have that 
$\theta(p^+)-\theta(p^-) = \pm  2\pi n \theta_i $ modulo $2\pi \Z$. We then have that there exists a constant $c_i>0$ such that
$$
\begin{aligned}
\int_{\hat A_i} \Big(\vert\hat Y \wedge \hat \Theta\vert^2 + \vert\hat X \wedge &\hat \Theta\vert^2 \Big)^{1/2}  \geq c_i\Big\vert \int_{\hat A_i \setminus \hat I_i} \hat V \wedge d\theta \Big\vert =
c_i\Big\vert \int_{\partial (\hat A_i \setminus \hat I_i)} \theta\hat V  \Big\vert \\ & = c_i\Big\vert \int_{\hat I^+_i} \theta\hat V  - \int_{\hat I^-_i} \theta\hat V  \Big\vert=
 2\pi c_id_{\R} (n \theta_i, \Z) \int_{\hat I_i} \hat V\,.
\end{aligned}
$$
The statement of the lemma follows since the above inequality
holds for all $i\in \{1, \dots, d\}$.
\end{proof}

\begin{proof} [Proof of Theorem \ref{thm:sdestimate}] The
theorem follows from the limit in formula \eqref{eq:eigen_limit}, from the Cheeger-type lower bound for the eigenvalues in Theorem \ref{thm:Cheeger}, from the lower bound on the Cheeger constant in Theorem \ref{thm:Cheeger_const} and from the Definition \ref{def:Dioph_cond} of simultaneously Diophantine holonomy.

\end{proof}

As a corollary of the Cheeger constant estimate, we can show that the partial Laplacian $H$ has solutions $Hu = f$ if $f$ has zero average and is sufficiently regular in the fiber direction.

\begin{theorem}
\label{thm:H_apriori_est}
If the holonomy of the flat surface $(S,R)$ satisfies a simultaneous Diophantine  condition of exponent $\gamma >0$ (see Definition \ref{def:Dioph_cond}), then  for any zero-average function $f \in H^{r,s}(M)$ with $r\geq 0$ and $s \geq 2\gamma $, there is a unique zero-average solution $u \in H^{r,s- 2\gamma}(M)$ to the equation 
\[
Hu = f.
\]
In addition, for every $r\geq 0$ and $s \geq 2\gamma $ there exists a constant $C_{r,s}>0$ such that we have
$$
\vert u \vert_{r, s-2\gamma} \leq  C_{r,s} \vert  f \vert_{r,s}\,.
$$
 \label{Hbounded}
\end{theorem}
\begin{proof}
For every $n \in \Z$, let $\{\phi_{n,k}\}_{k\in \N}$ denote a orthonormal basis of eigenfunction of $H$ on the space $E_n$ and let $\{\lambda_{n,k}\}_{k\in \N}$ be the corresponding sequence of eigenvalues such that 
$$
H \phi_{n,k} = \lambda_{n,k} \phi_{n,k}\,.
$$
Let us assume that $\phi_{0,0}\equiv 1 \in E_0$. We can then write, for any $f \in L^2(M, \text{vol}_M)$, 
$$
f = \sum_{n\in \Z} \sum_{k>0} f_{n,k} \phi_{n,k}\,.
$$
We note that by the definition of the Sobolev norms, for 
every $r,s \geq 0$ there exists a constant $K_{r,s}>0$ such that,
for all $n\in \Z$ and $k\geq 0$ we have
$$
K_{r,s}^{-1}  (1+ n^2)^s \vert \phi_{n,k} \vert^2_{r,0} \,\leq\,  \vert \phi_{n,k} \vert^2_{r,s} \, \leq \, K_{r,s} (1+ n^2)^s \vert \phi_{n,k} \vert^2_{r,0} 
$$
The unique formal solution $u$ of zero average is then given by the formula
$$
u:=\sum_{n\in \Z} \sum_{k>0} f_{n,k} \lambda_{n,k}^{-1} \phi_{n,k}\,.
$$
We note that since the operator $H$ on $E_0$ is unitarily equivalent to the (non-negative) flat Laplacian on $S$ (which is elliptic and therefore has a spectral gap), it follows that there exists a constant $\lambda_S>0$ such that 
$$
\lambda_{0, k}  \geq \lambda_S >0\,, \quad \text{ for all } k>0\,.
$$
By the hypothesis of simultaneously Diophantine holonomy and
by Theorem \ref{thm:sdestimate} we derive that there exist
constants $C_s$ and $C'_s$ (also depending on $(S,R)$) such that
$$
\begin{aligned}
\vert u \vert^2_{r, s-2\gamma} &\leq C_s \sum_{n\in \Z} \sum_{k>0}  (1+ n^2)^{s-2\gamma} \vert f_{n,k} \vert^2 \vert \lambda_{n,k}\vert ^{-2} \vert \phi_{n,k} \vert^2_{r,0}
\\ &\leq  C_s \sum_{n\in \Z} \sum_{k>0}  (1+ n^2)^{s-2\gamma} \vert f_{n,k} \vert^2 \vert (1+ n^2)^{2\gamma} \vert \phi_{n,k} \vert^2_{r,0} \leq C'_s \vert f \vert^2_{r,s}\,.
\end{aligned}
$$

\end{proof}

\section{The cohomological equation}
\label{sec:CE}

In this section we finally derive our main results on solutions
of the cohomological equation for the flat geodesic flow and
its rotations.

\subsection{Distributional and Smooth Solutions}
\label{ssec:solutions}

In this subsection we prove a priori bounds for the Lie derivative operator $X$ of the flat geodesic flow on smooth functions and derive a result on distributional and 
smooth solutions of the cohomological equation.

The following a priori estimate for the foliated Cauchy-Riemann
derivatives in terms of the geodesic flow derivative holds.

\begin{lemma} 
\label{lemma:CR_apriori}
For every $r\geq 0$ and every $s > s'+ 3/2$ there exists a constant  $C_{r,s,s'}>0$ such that the following holds. 
For any $v \in H^{r+1,0}(M)$ such $Xv \in H^{r,s}(M)$ we have
that $\partial^\pm v \in H^{r,s'}(M)$ and 
$$
\vert \partial^\pm v \vert_{r,s'} \leq C_{r,s,s'} \vert Xv \vert_{r,s} \,.
$$
\end{lemma}
\begin{proof}
The equation  $Xv=f$, for $v \in H^{1,0}(M)$, is equivalent to the sequence of equations 
\begin{equation}
\label{eq:sol_id}
\partial^+ v_{n+1} + \partial^- v_{n-1} = 2f_n \,, \quad \text{ for all }n\in \Z\,.
\end{equation}
By taking into account that by Lemma \ref{lemma:CR_identity}
we have $\Vert \partial^+ v\Vert_{L^2(M, \text{vol}_M)} = \Vert \partial^- v\Vert_{L^2(M, \text{vol}_M)}$ for all $v\in H^{1,0}(M)$, it follows from the identity \eqref{eq:sol_id} that, for all $n\in \N$,
$$
\begin{aligned}
\Vert \partial^- v_{n-1} \Vert_{L^2(M, \text{vol}_M)} &\leq  2\Vert f_n \Vert_{L^2(M, \text{vol}_M)}  + \Vert \partial^+ v_{n+1}\Vert_{L^2(M, \text{vol}_M)}  \\ &= 2\Vert f_n \Vert_{L^2(M, \text{vol}_M)}  + \Vert \partial^- v_{n+1}\Vert_{L^2(M, \text{vol}_M)}\,,
\end{aligned}
$$
hence by induction
$$
\Vert \partial^- v_{n-1} \Vert_{L^2(M, \text{vol}_M)}
\leq 2\sum_{k=0}^{K} \Vert f_{n+2k} \Vert_{L^2(M, \text{vol}_M)}
+ \Vert \partial^- v_{n+2K+1} \Vert_{L^2(M, \text{vol}_M)}\,.
$$
Since  $\partial^- v\in L^2(M, \text{vol}_M)$ as  $v\in H^{1,0}(M)$
and 
$$
 \Vert \partial^- v\Vert^2_{L^2(M, \text{vol}_M)} = \sum_{n\in \Z}  \Vert \partial^- v_n\Vert^2_{L^2(M, \text{vol}_M)}\,,
 $$
it follows that $\Vert \partial^- v_{n+2K+1} \Vert_{L^2(M, \text{vol}_M)} \to 0$ as $K\in \N$ diverges. We therefore have the estimate
\begin{equation}
\label{eq:recursive_est_1}
\Vert \partial^- v_{n-1} \Vert_{L^2(M, \text{vol}_M)}
\leq 2\sum_{k=0}^{+\infty} \Vert f_{n+2k} \Vert_{L^2(M, \text{vol}_M)}
\end{equation}
Similarly, we have 
$$
\begin{aligned}
\Vert \partial^+ v_{n+1} \Vert_{L^2(M, \text{vol}_M)} &\leq  2\Vert f_n \Vert_{L^2(M, \text{vol}_M)}  + \Vert \partial^- v_{n-1}\Vert_{L^2(M, \text{vol}_M)}  \\ &= 2\Vert f_n \Vert_{L^2(M, \text{vol}_M)}  + \Vert \partial^+ v_{n-1}\Vert_{L^2(M, \text{vol}_M)}\,,
\end{aligned}
$$
hence
\begin{equation}
\label{eq:recursive_est_2}
\Vert \partial^+ v_{n+1} \Vert_{L^2(M, \text{vol}_M)}
\leq 2\sum_{k=0}^{+\infty} \Vert f_{n-2k} \Vert_{L^2(M, \text{vol}_M)}
\end{equation}
Let us assume that $f \in H^{0,s}(M)$. It follows that 
$$
\Vert f_n \Vert_{L^2(M, \text{vol}_M)} \leq \vert f \vert_{0,s}
(1+n^2)^{-s/2} \,,
$$
hence by estimates \eqref{eq:recursive_est_1} for $n\geq 0$ and
\eqref{eq:recursive_est_2} for $n\leq 0$ we derive that, for 
all $ s>1$ there exists a constant $C_s >0$ such that 
$$
\Vert \partial^+ v_n \Vert_{L^2(M, \text{vol}_M)} =
\Vert \partial^- v_n \Vert_{L^2(M, \text{vol}_M)} \leq 
C_s \vert f \vert_{0,s} (1+n^2)^{-(s-1)/2}\,.
$$
so that, for $s>3/2$, there exists a constant $C'_s>0$ such that
$$
\Vert \partial^+ v \Vert_{L^2(M, \text{vol}_M)} =
\Vert \partial^- v\Vert_{L^2(M, \text{vol}_M)} \leq C'_s 
\vert f \vert_{0,s}\,.
$$
In fact, more generally for every $s'\geq 0$ and for $s > s' + 3/2$ there exists a constant $C_{s,s'}>0$ such that
$$
\vert \partial^+ v \vert_{0,s'} =
\vert \partial^- v \vert_{0,s'} \leq C_{s,s'} 
\vert f \vert_{0,s}\,.
$$
Since $[H, X]=0$ on $H^{3,0}(M)$, it follows that under the
assumption that $ v \in H^{2k+1,0}(M)$ we have
$$
X (H^k v) = H^k (Xv) = H^k (f)
$$
so that,  if  $f \in H^{2k, s}(M) $ with $s > s' +3/2$, we have
$$
\vert H^k (\partial^\pm v) \vert_{0,s'} =
\vert  \partial^\pm H^k(v)\vert_{0,s'} \leq C_{s,s'} \vert f \vert_{2k,s} \,,
$$
which immediately implies the conclusion for $r=2k$ even.
The statement for odd $r=2k+1$ follows by interpolation.
\end{proof}

\begin{theorem}
\label{thm:aprioriest_smooth}
Let us assume that the holonomy of the flat surface $(S,R)$ satisfies a simultaneous Diophantine  condition of exponent $\gamma >0$.  Given $r',s'\geq 0$ and for every $r >r' +1$ and $s>s' + 2\gamma +3/2$, there exists a constant $C:= C(r,r's,s',\gamma)>0$ such that, for all zero average functions $v \in H^{r,0}(M)$ with $Xv \in H^{r,s}(M)$ we have that 
\begin{equation}
\vert v\vert _{r',s'} \leq C \vert X v\vert _{r,s} \,.
\end{equation}
\end{theorem}
\begin{proof}
By Lemma \ref{lemma:CR_apriori} we have
$$
\vert \partial^\pm v \vert_{r'+1,s'+2\gamma} \leq C_{r,s,s'} \vert Xv \vert_{r,s}\,,
$$
hence
$$
\vert H v \vert_{r',s'+2\gamma} \leq C_{r,s,s'} \vert Xv \vert_{r,s}\,.
$$
By Theorem \ref{thm:H_apriori_est} under the assumption that 
$v$ has zero average we have, for any $r',s'\geq 0$ we derive the estimate
$$
\vert v \vert_{r', s'} \leq C_{r',s'} \vert H v \vert_{r',s'+2\gamma} \leq C_{r,s,s'} C_{r,s,s'+2\gamma} \vert Xv \vert_{r,s}\,,
$$
as stated.
\end{proof}

For all $r,s \geq 0$, let $\mathcal I^{r,s}_X(M) \subset 
H^{-r,-s}(M)$ denote the space of all $X$-invariant distributions:
$$
\mathcal I^{r,s}_X(M):=\{ D \in H^{-r,-s}(M) \vert XD=0\}\,.
$$
We recall that the distributional equation $XD=0$ for a distribution $D\in H^{-r,-s}(M)$ is by definition equivalent to the condition 
$$
D(Xv) = 0 \,, \quad \text{ for all } v\in H^{r+1,s}(M) \,.
$$
We prove below our main theorem about distributional and smooth solutions of the cohomological equation.

\begin{theorem}
\label{thm:CE_dist_sol}
Let us assume that the holonomy of the flat surface $(S,R)$
satisfies a simultaneous Diophantine  condition of exponent $\gamma >0$.  Given $r',s'\geq 0$ and for every $r >r' +1$ and $s>s' + 2\gamma +3/2$, there exists a constant $C:= C(r,r's,s',\gamma)>0$ such that, for any distribution  $F \in H^{-r',-s'}(M)$ vanishing on constant functions there exists a solution $U\in H^{-r,-s}(M)$ of the cohomological equation $XU =F$ such that 
\begin{equation}
\vert U\vert _{-r,-s} \leq C \vert F\vert _{-r',-s'} \,.
\end{equation}
\end{theorem} 
\begin{proof}
Given a distribution $F\in H^{-r',-s'}(M)$ vanishing on constant function, we can formally define a solution $U\in H^{-r,-s}(M)$ on the range of the Lie derivative operator by the formula
$$
\langle U, Xv \rangle = -\langle F, v\rangle\,, \quad \text{ for all } v \in H^{r+1,s}(M)\,.
$$
By Theorem~\ref{thm:aprioriest_smooth} we have, for all $v\in H^{r+1,s}(M)$, 
$$
\vert \langle U, Xv \rangle \vert = \vert \langle F, v\rangle\vert \leq \vert F\vert_{-r',-s'} \vert v \vert_{r',s'}
\leq C_{r,s,r',s'} \vert F\vert_{-r',-s'} \vert Xv \vert_{r,s}\,,
$$
hence the formal solution $U$ extends by continuity (since the Sobolev space $H^{r,s}(M)$ is complete) to a bounded linear
functional on the closed subspace 
\begin{equation}
\label{eq:R_rs}
R_{r,s}(X):= \overline{ \{ f \in H^{r,s}(M) \vert f =Xv  \text{ with } v \in C^\infty_\Sigma(M)\}}^{ H^{r,s}(M)} \,,
\end{equation}
and to a bounded linear functional on $H^{r,s}(M)$ defined
as $0$ on the orthogonal complement $R_{r,s}(X)^\perp \subset H^{r,s}(M)$.
It follows by its definition that the distribution 
$U \in H^{-r,-s}(M)$ is a solution of the cohomological equation $XU=F$ and satisfies the stated bound.
\end{proof}

\begin{theorem}
\label{thm:CE_smooth_sol}
Let us assume that the holonomy of the flat surface $(S,R)$ satisfies a simultaneous Diophantine  condition of exponent $\gamma >0$. Given  $r>r'+1 \geq 2$ and $s > s' + 3/2+2\gamma \geq 3/2+2\gamma$, there exists a constant $C_{r,s,r',s'}>0$
such that the following holds.  For every $f \in H^{r,s}(M)$
such that $f \in \text{\rm Ker} (\mathcal I^{r,s}_X(M) )$, there
exists a unique zero average solution $u\in H^{r',s'}(M)$ of the cohomological equation $Xu =f$ which satisfies the estimate
$$
\vert u \vert_{r',s'} \leq C_{r,s,r',s'} \vert f \vert_{r,s}\,.
$$
\end{theorem}
\begin{proof}
Let $R_{r,s}(X)\subset H^{r,s}(M)$ the closure of the range of the Lie derivative operator $X$ as in formula \eqref{eq:R_rs}, that is, 
$$
R_{r,s}(X):= \overline{ \{ f \in H^{r,s}(M) \vert f =Xv  \text{ with } v \in C^\infty_\Sigma(M)\}}^{ H^{r,s}(M)} \,.
$$
We claim that the subspace $R_{r,s}(X)$ is equal to the
joint kernel $\text{\rm Ker} (\mathcal I^{r,s}_X(M) )$ of all $X$-invariant distributions in the dual  space  $H^{-r,-s}(M)$.
In fact, on the one hand if $D \in \text{\rm Ker} (\mathcal I^{r,s}_X(M) )$, then by definition $D(Xv)=0$ for all $v\in C^\infty_\Sigma (M)$, hence $D$ vanishes on $R_{r,s}(X)$
since $D$ is a continuous functional on $H^{r,s}(M)$. On the
other hand, by the Hahn-Banach theorem, the closed subspace $R_{r,s}(X)$ is equal to the joint kernel of all functionals
$D \in H^{-r,-s}(M)$ which vanish on it. However, any functional
vanishing on $R_{r,s}(X)$, vanishes on all functions of the form
$Xv$ with $v\in H^{r+1,s}(M)$, hence it is an $X$-invariant distribution. The claim is therefore proved.

Finally, we prove that for all $f \in R_{r,s}(X)$, there exists
a solution $u \in H^{r',s'}(M)$ satisfying the stated estimates.
In fact, since $f \in R_{r,s}(X)$ there exists a sequence $\{f_k\}$, such that $f_k =X v_k$ with $v_k \in C^\infty_\Sigma(M)$ of zero average, which converges to $f$ in $H^{r,s}(M)$.

By Theorem \ref{thm:aprioriest_smooth} it follows that, for all
$k, h\in \N$, 
$$
\vert v_k \vert_{r',s'} \leq C \vert f_k \vert_{r,s} \quad 
\text{ and } \quad \vert v_k -v_h \vert_{r',s'} \leq C \vert f_k -f_h\vert_{r,s}\,,
$$
so that, since $\{f_k\}$ is convergent, hence a Cauchy sequence in $H^{r,s}(M)$, then $\{v_k\}$ is a Cauchy sequence in $H^{r',s'}(M)$. By completeness the sequence $\{v_k\}$ converges
to a function $u \in H^{r',s'}(M)$ which is a solution of the
cohomological equation $Xu =f$ and satisfies the stated bound
by continuity.

Finally, we conclude with the observation that since $r'\geq 1$
and $s'\geq 0$, by Corollary \ref{cor:no_smooth_inv} the zero-average solution of the cohomological equation in $H^{r',s'}(M)$
is unique. The same conclusion can also be derived from the
a priori estimate proved in Theorem \ref{thm:aprioriest_smooth}.
\end{proof}

\subsection{Invariant distributions}
\label{ssec:Inv_dist}
In this sections we prove some results on the space of invariant distributions for the geodesic flow, in particular we prove that
it has countable dimension and codimension. 

The space of invariant distributions is related to the space
of foliated meromorphic (or  anti-meromorphic) functions through the following constructions.

Let us assume that, for a given meromorphic functions $m$ on $M$ (square-integrable or otherwise taken as a distribution), the
cohomological equation $XU =m$ has a distributional solution.
Then the distribution $D= \partial^+ U$ is an $X$-invariant
distribution (possibly equal to the zero distribution). 

In this section we make this outline precise. We begin by
a description of the spaces of foliated meromorphic and anti-meromorphic distributions.

For every $r\in \N$ and $n\in \Z$, let 
$$
E_n^{-r} = \{ D \in H^{-r,0}(M) \vert \Theta(D) = in D\}\,.
$$

\begin{theorem} 
\label{thm:meromorphic}
For all $n\in \Z$ the subspaces $\mathcal M^{+,r}_n \subset E^{-r}_n$ and $\mathcal M^{-,r}_{-n} \subset E^{-r}_{-n}$ of foliated meromorphic and anti-meromorphic distributions are in bijective correspondence respectively
with the space of meromorphic  $n$-differentials on $S$ and
the space of anti-meromorphic $n$-differentials on $S$ with order of zero or pole $k_p$ at any conical point $p\in \Sigma$ of parameter $\alpha_p$, and angle $2\pi(\alpha_p+1)$, satisfying the inequalities
$$
k_p > (n-1)\alpha_p - r- 1 \,.
$$
\end{theorem}
\begin{proof} An $n$-differential $d_n$ is a section of the 
$n$-th power of the canonical bundle $T^\ast (S)$ of the Riemann surface $S$. It induces a function on $M= T_1(S\setminus \Sigma)$ by  the formula
\begin{equation}
\label{eq:n-diff}
f_n(x,v) = d_n(x) (v) \,, \quad \text{ for all }(x,v) \in M\,.
\end{equation}
The function $f_n$ is well-defined on $M$, it is foliated
meromorphic (anti-meromorphic) on $M$ if $d_n$ is meromorphic
(anti-meromorphic) and there exists $r\geq 0$ such that $f_n$ belongs to $E_n^{-r}$ since $d_n$ is an $n$-differential, hence
$$
\Theta f_n = in f_n\,.
$$
Conversely, given a foliated meromorphic (anti-meromorphic) 
function $f_n$ on $M$ which verifies the above differential
equation, the identity in formula \eqref{eq:n-diff} uniquely
defines a meromorphic (anti-meromorphic) $n$-differential $d_n$.

Clearly any meromorphic or anti-meromorphic $n$-differential induces
a foliated meromorphic, respectively anti-meromorphic, function on $M$. We then investigate the relation between the integrability properties of the induced function 
and the order of zero or pole of the differential. 

Let $p\in \Sigma$ be a conical point of total angle $2\pi(\alpha+1)$. We recall that the metric on $S$ near $p$
has the expression $\vert z^\alpha dz \vert$ (with respect
to a canonical coordinate). Let $d_n= d_n(z) dz^n$  be a meromorphic  differential with a zero or pole of order $k\in \Z$ at $p$ ($k>0$ for a zero and $k<0$ for a pole). The function $f_n$ has the form
$$
f_n(z,\theta) = d_n(z) (z^{-\alpha} e^{i\theta} )^n
$$
hence by integrating on a disk centered at $p$ in polar coordinate we can verify that $f_n$ is locally square integrable iff
$$
\int_0^1  r^{2k} r^{-2n\alpha} r^{2\alpha +1} dr < +\infty\,.
$$
that is, iff
$$
k > (n-1)\alpha -1 \,.
$$
We conclude that the above inequality for all $p\in \Sigma$ is the necessary and sufficient condition for the function $f_n$
to belong to the space $\mathcal M^{+,0}_n$ of square-integrable foliated meromorphic eigenfunctions of $\Theta$ with eigenvalue $in$. 

A similar set of inequalities for all $p\in \Sigma$ is the necessary and sufficient condition for the function $f_n$ induced by an anti-meromorphic $n$-differential $d_n$
to belong to the space $\mathcal M^{-,0}_{-n}$ of square-integrable foliated anti-meromorphic eigenfunctions of $\Theta$ with eigenvalue $-in$. 

We recall that the Cauchy-Riemann operators are given 
in a local canonical coordinate near a cone point of 
angle $2\pi(\alpha+1)$ by the formulas
$$
\partial^+ = e^{-i\theta}\bar{z}^{-\alpha} \frac{\partial}{\partial \bar{z}} \,, \quad \partial^- = e^{i\theta}z^{-\alpha} \frac{\partial}{\partial z} \,.
$$
Let $f_{n-1}$ be the foliated meromorphic function on $M$ given by a meromorphic $(n-1)$-differential $d_{n-1}$. We then have 
$$
\begin{aligned}
\partial^- f_{n-1}(z,\theta) &= e^{i\theta}z^{-\alpha} \frac{\partial}{\partial z} \Big(d_{n-1}(z) (z^{-\alpha} e^{i\theta} )^{n-1}\Big) \\ &=  \frac{\partial d_{n-1}}{\partial z}(z) (z^{-\alpha} e^{i\theta} )^{n}
- (n-1)\alpha  d_{n-1}(z) z^{-1} (z^{-\alpha} e^{i\theta} )^{n}\,,
\end{aligned}
$$
which is the foliated holomorphic function induced by the meromorphic $n$-differential 
\begin{equation}
\label{eq:CR_n_diff}
d_n = \Big(\frac{\partial d_{n-1}}{\partial z}(z)- (n-1) \alpha d_{n-1}(z) z^{-1} \Big) dz^{n}\,,
\end{equation}
hence the (non-negative) order of pole is changed by $-1$.
A similar calculation can be carried out for the foliated
anti-meromorphic functions (or it can be derived by complex
conjugation). 

Since the the range of the map $$d_{n-1}(z) \to \partial d_{n-1}(z)/\partial z- (n-1)\alpha d_{n-1}(z) z^{-1}$$ on the space of Laurent series consists of all Laurent series (with no term of order $(n-1)\alpha-1$ in case $(n-1)\alpha \in \Z$), it follows that, locally near $p\in \Sigma$, all functions induced by a meromorphic $n$-differential with order of zero $k\in \Z$, with no term of order $(n-1)\alpha-1$ in case $(n-1)\alpha \in \Z$, are $\partial^-$ derivatives of functions induced by a meromorphic $(n-1)$-differential with order of zero $k+1$. In case $(n-1)\alpha \in \Z$, the term
of the Laurent series of order $(n-1)\alpha -1$ is given by the
identity 
$$
z^{(n-1)\alpha -1} = \frac{\partial d_{n-1}}{\partial z}(z) - n\alpha d_{n-1}(z) z^{-1} \quad \text {for } d_{n-1}(z) = z^{(n-1)\alpha} \log \vert z \vert\,.
$$
The   differential $z^{(n-1)\alpha} \log \vert z \vert dz^{n-1}$ (which is (neither meromorphic nor anti-meromorphic) gives a square integrable function, hence
$z^{(n-1)\alpha-1}  dz^{n}$ gives a function in $\mathcal M^{+,-1}$.

\noindent Similar statements hold for anti-meromorphic functions induced by  anti-meromorphic $n$-differentials.

\noindent We conclude that the necessary and sufficient condition for the function $f_n$ induced by a meromorphic $n$-differential $d_n$ to belong to the space $\mathcal M^{+,r}_n$ of  foliated meromorphic functions in $E^{-r}_n$ are given by the
set of inequalities
$$
k_p > (n-1)\alpha_p - r - 1 \,, \quad \text{ for all } p\in \Sigma.
$$
A similar set of inequalities for all $p\in \Sigma$ is the necessary and sufficient condition for the function $f_n$ induced by an anti-meromorphic $n$-differential 
to belong to the space $\mathcal M^{-,r}_{-n}$ of foliated anti-meromorphic functions in $E^{-r}_{-n}$. 
\end{proof}

\begin{corollary} 
\label{cor:mero_sq_int}
For every flat surface $(S,R)$ with conical singularities, there exist infinite (countable) sets $Z^\pm_S\subset \Z$ such that 
$$
\mathcal M^{\pm, 0}_n \not=\{0\} \,, \quad \text{ for all } n\in Z^\pm_S\,.
$$
\end{corollary}
\begin{proof} 

By \cite[Theorem 2.1]{BCGGM}, given any integral
vector $(k_1, \dots, k_\sigma) \in \Z^\sigma$ such that 
$$
\sum_{i=1}^\sigma k_i = n (2g-2)\,,
$$
the moduli space of meromorphic $n$-differentials having $\sigma$ zeros or poles of orders $(k_1, \dots, k_\sigma)$ has dimension $2g-2 + \sigma$.

Let $(\alpha_1, \dots, \alpha_\sigma)$ denote the parameters
of the cone angles of $(S,R)$.
By the ergodic properties of toral translations, for every $\epsilon>0$ there exists an infinite subset $Z_S(\epsilon) \subset \Z$ such that, for all $n \in Z_S(\epsilon)$, 
$$
\min _{(k_1, \dots, k_\sigma) \in \Z^\sigma}  \vert (n-1) (\alpha_1, \dots, \alpha_\sigma)
- (k_1, \dots, k_\sigma) \vert < \epsilon \,.
$$
Let $(k'_1, \dots, k'_\sigma) \in \Z^\sigma$ be an integer
vector which realizes the above minimum. It follows that 
$$
\vert \sum_{i=1}^\sigma k'_i -(n-1) \sum_{i=1}^\sigma \alpha_\sigma \vert = \vert \sum_{i=1}^\sigma k'_i -(n-1) (2g-2) \vert < \sigma \epsilon\,.
$$
For $\sigma \epsilon <1$ we derive that
$$
\sum_{i=1}^\sigma k'_i = (n-1) (2g-2) \,.
$$
For $g \geq 2$, since $2g-2 \geq 0$, it follows that there exists $(k_1, \dots, k_\sigma)$ such that $\sum_{i=1}^\sigma k_i = n(2g-2)$ and, for all $i \in \{1, \dots, \sigma\}$,
$$
k_i  \geq k'_i > (n-1) \alpha_i - \epsilon > n-1) \alpha_i - 1\,.
$$
It then follows from Theorem \ref{thm:meromorphic} (with $r=0$) that, for all $n\in Z(\epsilon)$, the spaces $\mathcal M^{+,0}_n$ and $\mathcal M^{-,0}_{-n}$ have positive dimension. In fact, by the Riemann-Roch theorem, for $g\geq 2$ and $n>1$ the space of $n$-differentials with orders of zeros given by the vector $(k'_1, \dots, k'_\sigma)$ (or higher) is at least $$(2n-1)(g-1) - (n-1)(2g-2)= g-1\,.$$
In the case $g=1$, the surface has no conical singularities and
has integral holonomy, therefore all the spaces $\mathcal M^\pm_{n}$ have complex dimension $1$ (the only homolomorphic or anti-holomorphic functions on complex tori are constant).

\smallskip
In the case $g=0$ (and then  $\sigma \geq 3$), by a explicit computation, or by the Riemann-Roch theorem, the space of $n$-differentials with zeros or poles of order $(k_1, \dots, k_\sigma)$ has dimension 
$$
-2n - \sum_{i=1}^\sigma k_i + 1 \,.
$$
Given the ergodic properties of toral translation,  since
$$
\sum_{i=1}^\sigma [(n-1)\alpha_i -1] = -2n +2 -\sigma \,,
$$
there exists an infinite set $Z_S\subset \Z$ such that, for each $n \in Z_S$ there exists $(k_1, \dots, k_\sigma) \in \Z^\sigma$ with $k_i > (n-1)\alpha_i -1$ for all $i \in \{1, \dots, \sigma\}$ and 
$$
 -2n +2 -\sigma= \sum_{i=1}^\sigma [(n-1)\alpha_i -1]< \sum_{i=1}^\sigma k_i = -2n +3 -\sigma \,.
$$
It follows that for $n\in Z_S$ the dimension of $\mathcal M^{+,0}_n$ and $\mathcal M^{-,0}_{-n}$  is at least 
$$
-2n -(2n+3 -\sigma) +1 = \sigma-3 + 1 >0\,.
$$
The argument is complete.
\end{proof}

By Theorem \ref{thm:CE_dist_sol} we derive, as a special case, the
existence of distributional solutions of the cohomological equation with data given by foliated meromorphic or anti-meromorphic distributions vanishing on constant functions.

\begin{corollary}
\label{cor:mero_dist_sol}
Let us assume that the holonomy of the flat surface $(S,R)$
satisfies a simultaneous Diophantine  condition of exponent $\gamma >0$.  Given $r',s'\geq 0$ and for every $r >r' +1$ and $s>s' + 2\gamma +3/2$, there exists a constant $C:= C(r,r's,s',\gamma)>0$ such that, for all meromorphic, respectively  anti-meromorphic, distribution  $m^\pm \in H^{-r',-s'}(M)$, vanishing on constant functions there exists a solution $U^\pm \in H^{-r,-s}(M)$ of the cohomological equation $XU^\pm =m^\pm$ such that 
\begin{equation}
\vert U^\pm\vert _{-r,-s} \leq C \vert m^\pm\vert _{-r',-s'} \,.
\end{equation}
\end{corollary} 

The construction of $X$-invariant distributions proceeds as follows (in analogy with the case of translation surfaces \cite{Forni1997}). We define the following linear operators
$\mathcal D^\pm$ from the space of meromorphic distributions to the space $\mathcal I_X(M)$ of invariant distributions. 

Let $\mathcal M^{r',s',\pm}_0 \subset H^{-r',-s'}(M)$ denote
the subspace of meromorphic, respectively anti-meromorphic,
distributions vanishing to constant functions. Let 
$\mathcal I^{r,s}_X(M) \subset H^{-r,-s}(M)$ be the subspace
of all $X$-invariant distributions.

\begin{theorem}
\label{thm:Dpm}
Let us assume that the holonomy of the flat surface $(S,R)$
satisfies a simultaneous Diophantine  condition of exponent $\gamma >0$. Let $r',s'\geq 0$ and for every $r >r' +2$ and $s>s' + 2\gamma +3/2$. There exist bounded linear maps $\mathcal D^\pm: \mathcal M^{r',s',\pm}_0 \to \mathcal I^{r,s}_X(M) \subset H^{-r, -s}(M)$, defined as
follows: for every $m^\pm \in \mathcal M^{r',s',\pm}_0$ there
exists a unique $U^\pm \in H^{-(r-1),-s}(M)$, orthogonal to $\mathcal I^{r-1,s}_X(M)$, such that
$$
X U^\pm = m^\pm  \quad \text{ and }  \quad \mathcal D^\pm(m^\pm)= \partial^\pm U^\pm\,.
$$ 
For $r''>r+2$ and $s'' > s+2\gamma$ the  maps $\mathcal D^\pm$ on the spaces $\mathcal M^{r'', s'', \pm}$ are onto the subspace $\mathcal I^{r,s}_X(M)_0 \subset \mathcal I^{r,s}_X(M)$ of $X$-invariant distributions vanishing on constant functions.
\end{theorem}
\begin{proof} 
Let $m^\pm \in \mathcal M^{-r',-s', \pm}$ a meromorphic/anti-meromorphic, distribution vanishing on constant functions. Under the condition of Corollary \ref{cor:mero_dist_sol} there exists a solution $U^\pm \in H^{-(r-1),-s}(M)$ of the cohomological equation $XU^\pm = m^\pm$. 

The distribution $D^\pm := \partial U^\pm \in H^{-r, -s}(M)$
is $X$-invariant since (by commutation)
$$
X D^\pm = X \partial U^\pm = \partial^\pm (XU^\pm) = \partial^\pm m^\pm =0 \,.
$$
We observe that the solution of the cohomological equation
$XU^\pm = m^\pm$ is not uniquely defined, but defined only 
up to invariant distributions in the space $H^{-(r-1),-s}(M)$.
However, we can consider the unique solution orthogonal in 
the Hilbert space $H^{-(r-1),-s}(M)$ to the subspace $\mathcal I_X^{r-1,s}(M) \subset H^{-(r-1),-s}(M) $ of all $X$-invariant distributions. Such a solution minimizes the norm among all
solutions, therefore by the estimate in Corollary \ref{cor:mero_dist_sol}, we have
$$
\vert D^\pm \vert_{-r,-s}= \vert \partial^\pm U^\pm \vert_{-(r-1),-s} \leq \vert U^\pm \vert_{-(r-1),-s} \leq C \vert m^\pm \vert_{-r',-s'} \,,
$$
hence the operator (which is easily seen to be linear) is bounded,
as stated.

It remains to prove that the maps $\mathcal D^\pm$ are onto
the subspace $\mathcal I^{r,s}_X(M)_0$.

Let $D \in \mathcal I^{r,s}_X(M)_0$. We first prove that 
the equations $\partial^\pm U^\pm = D $ have distributional solutions. In fact, by Theorem \ref{Hbounded}, for all zero-average functions $v \in H^{r+2, s+2\gamma}(M)$ we have 
$$
\vert \langle D,  v \rangle  \vert \leq \vert D \vert_{-r,-s}
\vert v \vert_{r,s} \leq C_{r,s} \vert H v \vert_{r, s+2\gamma}
\leq C_{r,s} \vert \partial^\pm v \vert_{r+1, s+2\gamma} \,,
$$
hence we can define a distribution $U^\pm$ by 
$$
\langle U, \partial^\pm v \rangle = - \langle D,  v \rangle \,,
\quad \text{ for all } v \in \in H^{r+2, s+2\gamma}(M)\,,
$$
the extended to the whole space $H^{r+1, s+2\gamma}(M)$ first by
continuity for the extension to the closure of the range of the foliated Cauchy-Riemann operator in $H^{r+1, s+2\gamma}(M)$, then by orthogonality (as zero on the orthogonal complement).
We have thus constructed a solution $U^\pm \in 
H^{-(r+1), -(s+2\gamma)}(M) $ of the equation $\partial^\pm U^\pm = D $. Let then $$m^\pm = X U^\pm \in H^{-(r+2), -(s+2\gamma)}(M)\,.$$ We claim that $m^\pm \in \mathcal  M^{r+2, s+2\gamma, \pm}_0$ and that $\mathcal D^\pm (m^\pm) =D$.  In fact,
$$
\partial^{\pm} m^\pm = \partial^\pm (XU^\pm )= X (\partial^\pm U^\pm) =XD =0\,,
$$
hence $m^\pm \in  \mathcal  M^{r+2, s+2\gamma, \pm}$ and vanish on constant functions since $m^\pm$ are derivatives. 
Finally, for any $r'' >r+2$ and $s'' >s+2\gamma$, the distribution $m^\pm \in \mathcal M^{r'',s'',\pm}_0$ and 
$\mathcal D^\pm (m^\pm) =D$ by the definition of the maps.
The argument is thus complete.
\end{proof}

\begin{corollary}
\label{cor:CE_inv_dist}
Let us assume that the holonomy of the flat surface $(S,R)$
satisfies a simultaneous Diophantine  condition of exponent $\gamma >0$. For all $r>2$ and $s>3/2 +\gamma$, the space $\mathcal I^{r,s}_X \subset H^{-r,-s}(M)$ of $X$-invariant distributions has countable dimension.
\end{corollary}
\begin{proof}
It suffices to prove that the kernel of the maps
$$\mathcal D^\pm: \mathcal M^{r',s',\pm}_0 \to \mathcal I^{r,s}_X(M) \subset H^{-r, -s}(M)$$ have countable codimension.
Let $m^\pm \in \text{Ker} (\mathcal D^\pm) \cap H^{-r',-s'}(M)$. By the definition of the maps $\mathcal D^\pm$, there exists a distribution $U^\pm \in H^{-(r-1), -s}(M)$ such that $\partial^\pm U^\pm =0$ and $X U^\pm = m^\pm$. It follows that
$U^\pm$ is a foliated meromorphic/ant-meromorphic function and
$$
m^\pm = X U^\pm = \frac{1}{2}(\partial^+ + \partial^-) U^\pm =
\frac{1}{2} \partial^\mp U^\pm\,. 
$$
In summary if $m^\pm \in \text{Ker} (\mathcal D^\pm) \cap H^{-r',-s'}(M)$, then $m^\pm \in \partial ^\mp (\mathcal M^{(r-1),s,\pm})$. It then suffices to prove that the range of Cauchy-Riemann operators in the spaces of all foliated meromorphic/anti-meromorphic distributions has countable codimension. As we have remarked above (see formula \eqref{eq:CR_n_diff}) on all the spaces $\mathcal M^{\pm,r}_n$ of meromorphic/anti-meromorphic distributions which belong the distributional eigenspace $E^{-r}_n$ of the operator $\Theta$, the Cauchy-Riemann operators lower by one the order of zero of the corresponding $n$-meromorphic/anti-meromorphic differential. Therefore, for each $n\in \Z$, there exists $r_n>0$ such that the $n$-meromorphic/anti-meromorphic differential of maximal total order of zero gives a foliated meromorphic/anti-meromorphic distribution $m^\pm_n \in \mathcal M^{\pm,r_n}_n$ which does not belong to the range of the Cauchy-Riemann operator $\partial^\mp$, hence the corresponding invariant distributions 
$D^\pm_n:=\mathcal D^\pm (m^\pm_n)$ are non-zero.  The set of distributions $\{D^\pm_n\}$ is linearly independent, since otherwise there would be a linear combination of the set 
$\{m^\pm_n\}$  which belongs to the range of the Cauchy-Riemann operator $\partial^\mp$. Since the Cauchy-Riemann operators respect the spectral decomposition of $\Theta$, it would follow
that each distribution $m^\pm_n$ in the linear combination belongs to the range of the Cauchy-Riemann operator $\partial^\mp$, in contradiction with the assumption that it has
maximal order of zero.

The above argument proves that the space $\mathcal I_X$ of all
$X$-invariant distributions has countable dimension. The proof
that there exists $r,s >0$ such that the space $\mathcal I^{r,s}_X$ of {\it finite order} $X$-invariant distributions has
countable dimension is based on Corollary \ref{cor:mero_sq_int}.
In fact, by that corollary there exist infinite sets $\Z^\pm_S\subset \Z$ such that $\mathcal M^{\pm,0}_n \not=\{0\}$.
By Theorem \ref{thm:Dpm}, the $X$-invariant distributions in $\mathcal D^\pm (\mathcal M^{\pm,0}_n) $ belong to the dual
Sobolev space $H^{-r,-s}(M)$ for all $r>2$ and $s>2\gamma+3/2$,
hence $\mathcal I^{r,s}_X$ has countable dimension.
\end{proof}

{
\addcontentsline{toc}{chapter}{References}
\bibliographystyle{plain}
\bibliography{references}
}
 
\end{document}